\documentclass[12pt]{amsart}

\usepackage{environ, xcolor}
\NewEnviron{commentA}{}

\newcommand\Aoff{\RenewEnviron{commentA}{}}
\Aoff{} 

\usepackage{etex}
\usepackage{amsmath, amssymb}
\usepackage{array}
\usepackage[frame,cmtip,arrow,matrix,line,graph,curve]{xy}
\usepackage{graphpap, color, paralist, pstricks}
\usepackage[mathscr]{eucal}
\usepackage[pdftex]{graphicx}
\usepackage[pdftex,colorlinks,backref=page,citecolor=blue]{hyperref}
\usepackage{tikz, verbatim}

\newtheorem{theorem}{Theorem}[section]
\newtheorem{proposition}[theorem]{Proposition}

\theoremstyle{definition}

\newtheorem{question}[theorem]{Question}

\usepackage[margin=1in]{geometry} 
\usepackage{amsmath,amssymb,mathtools, mathrsfs,esint,tikz-cd,verbatim}

\newcommand{\Z}{{\mathbb Z}}

\newcommand{\C}{{\mathbb C}}

\newcommand{\Hom}{{\text{Hom}}}

\newcommand{\Spec}{{\text{Spec}}}

\newcommand{\Pic}{{\textup{Pic}}}

\newcommand{\Br}{\textup{Br}}
\newcommand{\Cl}{\textup{Cl}}

\hypersetup{%
pdftitle={Non-torsion Brauer groups in positive characteristic},%
pdfauthor={Louis Esser},%
pdfkeywords={Brauer group, elliptic singularities},%
citecolor=blue,%
linkcolor=blue,%
}

\title{Non-torsion Brauer groups in positive characteristic}
\author{Louis Esser}

\address{UCLA Mathematics Department, Box 951555, Los Angeles, CA 90095-1555}
\email{esserl@math.ucla.edu}

\subjclass[2020]{14F22}
\keywords{Brauer groups, \'{e}tale cohomology}

\begin{document}
\maketitle

\begin{abstract}
    Unlike the classical Brauer group of a field, the Brauer-Grothendieck group of a singular scheme need not be torsion. We show that there exist integral normal projective surfaces over a large field of positive characteristic with non-torsion Brauer group. In contrast, we demonstrate that such examples cannot exist over the algebraic closure of a finite field.
\end{abstract}

\section{Introduction}
One way of extending the notion of the classical Brauer group of a field to any scheme $X$ is by defining the Brauer-Grothendieck group $\Br(X) = H_{\text{\'{e}t}}^2(X,\mathbb{G}_m)$.  Just as for fields, this group is torsion for any regular integral noetherian scheme \cite[Corollaire 1.8]{grothendieck}.  However, this no longer holds for singular schemes.  Chapter 8 of \cite{brauer} lists several counterexamples.  For instance, if $R$ is the local ring of the vertex of a cone over a smooth curve of degree $d \geq 4$ in $\mathbb{P}^2_{\C}$, then $\Br(R)$ contains the additive group of $\C$.  There are then affine Zariski open neighborhoods of the vertex with non-torsion Brauer group \cite[Example 8.2.2]{brauer}.  However, the additive group of $k$ is torsion when $k$ has positive characteristic, so analogous constructions do not work there.  There are also reducible varieties in arbitrary characteristic with non-torsion Brauer group \cite[Section 8.1]{brauer}.  This leaves open the following question, communicated by Colliot-Th\'{e}l\`{e}ne and Skorobogatov in an unpublished draft of \cite{brauer}:

\begin{question}
If $X$ is an integral normal quasi-projective variety over a field $k$ of positive characteristic, is $\Br(X)$ a torsion group?
\end{question}

To analyze this question, we use a result concerning the Brauer group of a normal variety $X$ with only isolated singularities $p_1,\ldots,p_n$.  Suppose $X$ is defined over an algebraically closed field $k$ of arbitrary characteristic.  Let $K$ be the function field of $X$, $R_i = \mathcal{O}_{X,p_i}$ be the local ring at each singularity, and $R_i^{\text{h}}$ its henselization.  Then, $\Br(X)$ is given by the exact sequence (see \cite[Section 8.2]{brauer}, elaborating on \cite[\S 1, Remarque 11 (b)]{grothendieck}) 

\begin{equation}
0 \rightarrow \Pic(X) \rightarrow \Cl(X) \rightarrow \bigoplus^n_{i=1} \Cl(R_i^{\text{h}}) \rightarrow \Br(X) \rightarrow \Br(K).
\end{equation}

This sequence indicates that one way for $\Br(X)$ to be large is for a singularity to have a large henselian local class group with divisors that do not extend globally to $X$.  This idea is illustrated by a counterexample given by Burt Totaro, which shows that $\Br(X)$ is non-torsion for $X$ a hypersurface of degree $d \geq 3$ in $\mathbb{P}_k^4$ with a single node \cite[Proposition 8.2.3]{brauer}.  Here $k$ is any algebraically closed field with characteristic not $2$.

\begin{commentA}
The below is a summary of the example due to Totaro, which was included in previous versions of the paper.

suppose that $X \subset \mathbb{P}^4$ is a hypersurface of degree $d \geq 3$ with a single node $p$.  Then $Y = \text{Bl}_p(X)$ is a smooth, ample divisor in $\text{Bl}_p \mathbb{P}^4$.  By the Grothendieck-Lefschetz theorem for Picard groups, the restriction $\Pic(\text{Bl}_p \mathbb{P}^4) \rightarrow \Pic(Y)$ is an isomorphism.  

If $E$ is the exceptional divisor, then the sequence $0 \rightarrow \Z \cdot [E] \rightarrow \Pic(Y) \rightarrow \Cl(X) \rightarrow 0$ yields $\Cl(X) \cong \Z \cdot \mathcal{O}_X(1)$, so that the restriction map $\Cl(X) \rightarrow \Cl(R^{\text{h}})$ is zero.  However, since a threefold node is \'{e}tale-locally the cone over a smooth quadric surface, one can show that the henselian local class group contains a copy of $\Z$.  Thus, $\Br(X)$ is not torsion.

\end{commentA}

We will show that counterexamples to Question 1.1 exist in dimension $2$ if and only if $k$ is not the algebraic closure of a finite field.

\textbf{Acknowledgements.} I thank Burt Totaro for suggesting this problem to me and for his advice.  This work was partially supported by National Science Foundation grant DMS-1701237.

\section{A Surface Counterexample}
The following construction is taken from \cite{kollar}, Example 1.27.  Take a smooth cubic curve $D$ in the projective plane and a quartic curve $Q$ that meets it tranversally.  Let $Y = \text{Bl}_{q_1,\ldots,q_{12}} \mathbb{P}^2$, where $q_1,\ldots,q_{12}$ are the points of intersection.  On $Y$, the proper transform $C$ of $D$ satisfies $C^2 = -3$.  Unlike rational curves, not all negative self-intersection higher genus curves may be contracted to yield a projective surface.  Rather, the contraction might only exist as an algebraic space.  However, in this case, $C$ is contractible.

\begin{proposition}
There exists a normal projective surface $X$ and a proper birational morphism $Y \rightarrow X$ whose exceptional locus is exactly $C$.
\end{proposition}

\begin{proof}
The Picard group of $Y$ is the free abelian group on $H$, the pullback of a general line in $\mathbb{P}^2$, and the exceptional divisors $E_1,\ldots,E_{12}$.  Then, we claim that the line bundle $L := \mathcal{O}_Y(4H - \sum_i E_i) = \mathcal{O}_Y(H + C)$ defines a basepoint-free linear system on $Y$.  Indeed, no point outside of the $E_i$ can be a base point, and the proper transforms of both $D$ + a line and $Q$ belong to the linear system.  These don't intersect on the exceptional locus by the transversality assumption.  Therefore, the system defines a morphism from $Y$ to projective space with image $X'$.  This morphism is birational because we have an injective map $H^0(Y,\mathcal{O}_Y(H)) \hookrightarrow H^0(Y,L)$ and $\mathcal{O}_Y(H)$ is the pullback of a very ample line bundle on $\mathbb{P}^2$.

The exceptional locus of the morphism $Y' \rightarrow X'$ is precisely the union of the irreducible curves in $Y$ on which $L$ has degree zero. If $C'$ an irreducible curve with $C' \cdot (H + C) = 0$, clearly $C'$ is not supported on the $E_i$, or the intersection would be positive.  Therefore, $H \cdot C' > 0$, meaning $C \cdot C' < 0$.  This means $C' = C$, so the exceptional locus is the curve $C$, which is mapped to a point. Thus, $X'$ is a surface birational to $Y$ and $Y \rightarrow X'$ is an isomorphism away from $C$.  Finally, passing to the normalization $X$ of $X'$, we may assume $X$ is a normal projective surface; the normalization will also be an isomorphism away from $C$, and the image of $C$ will again be a point in $X$.
\end{proof}

The resulting singularity $p$ of $X$ has minimal resolution with exceptional set exactly $C$, a smooth elliptic curve.  Singularities satisfying this condition are \textit{simple elliptic singularities}.  Over $\C$, such singularities are completely classified.  A simple elliptic singularity with $C^2 = -3$ is known as an $\tilde{E}_6$ singularity, and is complex analytically isomorphic to a cone over $C$ \cite{saito}.  Here, we present a computation of the henselian local class group $\Cl(R^\text{h})$ of the singularity that works in any characteristic.

Consider the pullback of the desingularization $Y \rightarrow X$ to a ``henselian neighborhood":

\begin{equation*}
\begin{tikzcd}
Y^{\text{h}} \arrow[r] \arrow[d] & Y \arrow[d] \\
\Spec(R^{\text{h}}) \arrow[r] & X .
\end{tikzcd}
\end{equation*}

The scheme $Y^{\text{h}}$ is regular and $Y^{\text{h}} \setminus C \cong \Spec(R^{\text{h}}) \setminus \{\mathfrak{m}\}$, so we have an exact sequence
 $0 \rightarrow \Z \cdot [C] \rightarrow \Pic (Y^{\text{h}}) \rightarrow \Cl(R^{\text{h}}) \rightarrow 0$.  Here, the first map is injective since $\mathcal{O}_{Y^{\text{h}}}(C)$ has nonzero degree on $C$.  It suffices, therefore, to compute $\Pic (Y^{\text{h}})$.  To do so, we'll first consider infinitesimal neighborhoods of $C$ in $Y$.
 
 \begin{commentA}
 Detailed explanation of why $Y^{\text{h}}$ is a regular scheme: essentially, this is because henselization behaves well with regularity and tensor products.  We first base change to $R$; the Zariski local rings of $Y$ are regular, so this gives $Y_R \rightarrow \Spec(R)$ with $Y_R$ regular.  Then, to cover $Y^{\text{h}}$, one can consider affine opens in $Y_R$.  For any such $\Spec(S) \subset R$, we have a ring map $R \rightarrow S$ and the corresponding open in $Y^{\text{h}}$ is $R^{\text{h}} \otimes_R S$.  The local rings there can be identified with the henselizations of the local rings of $S$, which are regular (see, e.g. \cite[0BSK]{stacks}).
 \end{commentA}
 
The sequence of infinitesimal neighborhoods $C = C_1 \subset C_2 \subset \cdots$ is defined by powers of the ideal sheaf $\mathcal{I}_C$. Notably, these $C_n$ are the same regardless of whether we consider them inside $Y$ or inside the henselian neighborhood $Y^{\text{h}}$. The normal bundle to $C$ in $Y$ gives obstructions to extending line bundles to successive neighborhoods, but we'll show that all line bundles extend uniquely.  The group $\varprojlim_n \Pic(C_n)$ in the proposition below is also the Picard group of the formal neighborhood of $C$ in $Y$.
 
\begin{proposition}
The restriction map $\varprojlim_n \Pic(C_n) \rightarrow \Pic(C)$ is an isomorphism.
\end{proposition}

\begin{proof}
It's enough to show that the maps $\Pic(C_{n+1}) \rightarrow \Pic(C_n)$ are all isomorphisms for $n \geq 1$.  Each extension $C_n \subset C_{n+1}$ is a first-order thickening, since $C_n$ is defined in $C_{n+1}$ by the square-zero ideal sheaf $\mathcal{I}_C^n/\mathcal{I}_C^{n+1}$.  Associated to such a thickening is a long exact sequence in cohomology \cite[0C6Q]{stacks}

$$\cdots \rightarrow H^1(C,\mathcal{I}_C^n/\mathcal{I}_C^{n+1}) \rightarrow \Pic (C_{n+1}) \rightarrow \Pic (C_n) \rightarrow H^2(C,\mathcal{I}_C^n/\mathcal{I}_C^{n+1}) \rightarrow \cdots .$$

We may take the outer cohomology groups to be over $C$ since the underlying topological spaces are the same.  As sheaves of abelian groups on $C$, we have $\mathcal{I}_C^n/\mathcal{I}_C^{n+1} \cong \mathcal{O}_C(-nC)$, a multiple of the conormal bundle.  But $C^2 = -3$ in $Y$ so this last bundle has degree $3n > 0$.  Since $C$ is genus $1$, the higher cohomology of $\mathcal{O}_C(-nC)$ vanishes and $\Pic (C_{n+1}) \rightarrow \Pic (C_n)$ is an isomorphism for all $n$. 
\end{proof}

\begin{commentA}
Detailed explanation as to why outer cohomology groups are over $C$, and not a thickening: we consider every sheaf in the sequence
$$0 \rightarrow \mathcal{I}_C^n/\mathcal{I}_C^{n+1} \rightarrow \mathcal{O}_{C_{n+1}}^* \rightarrow \mathcal{O}_{C_n}^* \rightarrow 0$$
to be a sheaf of an abelian groups, and forget the module structure (this doesn't change cohomology).  All the $C_n$ have the same underlying topological space so these sheaves are all on that topological space.  All that changes is the ringed space structure.
\end{commentA}

Now, we need only compare $\varprojlim_n \Pic(C_n)$ to $\Pic(Y_h)$.  Using the Artin approximation theorem \cite[Theorem 3.5]{artin}, the map $\Pic(Y^{\text{h}}) \rightarrow \varprojlim_n \Pic(C_n)$ is injective with dense image.  However, the topology of the latter group is discrete in this setting because each group of the inverse limit is $\Pic(C)$.  Therefore, the map is surjective also and $\Pic(Y^{\text{h}}) \cong \Pic(C)$.

\begin{commentA}
Detailed explanation: The formulation of Artin approximation in \cite{artin} is as follows: let $A$ be a local henselian ring with maximal ideal $\mathfrak{m}$ ($A$ here is the henselization of a finite type algebra over the field $k$) and let $\hat{A}$ be its completion.  Then for any functor $F: \{A-\text{algebras}\} \rightarrow \{\text{sets}\}$ that is locally of finite presentation, any $\hat{\xi} \in F(\hat{A})$, and any integer $n$, there exists $\xi \in F(A)$ such that

$$\hat{\xi} \equiv \xi \mod \mathfrak{m}^n.$$

In the proof of theorem 3.5, it is noted that $H^1(X \times_{\Spec(A)} -, \mathbb{G}_m)$ is a functor locally of finite presentation for $X \rightarrow A$ proper.  Hence in our setting, we can approximate any line bundle on the formal neighborhood up to $C_n$.  Since all the Picard groups are the same, this demonstrates surjectivity.
\end{commentA}

\begin{theorem}
Let $k$ be an algebraically closed field that is not the algebraic closure of a finite field and $X$ be the surface defined in the proof of Proposition 2.1. Then, $\Br(X)$ is non-torsion. 
\end{theorem}

\begin{proof}
From the above, we have the identification $\Cl(R^{\text{h}}) \cong \Pic(Y^{\text{h}}) / \Z \cdot \mathcal{O}_{Y^{\text{h}}}(C) \newline \cong \Pic(C) / \Z \cdot \mathcal{O}_C(C)$.  Since $\deg_C \mathcal{O}_C(C) = 3$, the class group is then an extension of $\Z/3$ by $\Pic^0(C) \cong C(k)$, where $C(k)$ is the group of $k$-rational points of the elliptic curve $C$ with a chosen identity point.  Since $k \neq \overline{\mathbb{F}}_p$, $C(k)$ has infinite rank \cite[Theorem 10.1]{frey}. Note that in contrast, $C(k)$ is torsion for an elliptic curve $C$ over $\overline{\mathbb{F}}_p$, because every point is defined over $\mathbb{F}_{p^m}$ for some $m$.

However, the global class group of $X$ is quite small: $\Cl(Y) = \Pic(Y) \cong \Z^{13}$ since $Y$ is the blow up of $\mathbb{P}^2$ in $12$ points and $\Cl(X) \cong \Cl(Y) / \Z \cdot [C]$.  Therefore, the cokernel of the restriction map $\Cl(X) \rightarrow \Cl(R^{\text{h}})$ in (1) contains non-torsion elements, so $\Br(X)$ does too.
\end{proof}

To complement the above result, we also prove:

\begin{theorem}
Suppose that $X$ is an integral normal surface over the algebraic closure $k = \overline{\mathbb{F}}_p$ of a finite field.  Then $\Br(X)$ is torsion.
\end{theorem}

\begin{proof}
The strategy is similar to the previous theorem. Here, the crucial fact is that all possible ``building blocks" of the henselian local class group - abelian varieties over $k$, the additive group of $k$, and the multiplicative group of $k$ - are torsion.

Since the singularities of a normal surface are isolated, we may apply the exact sequence (1). The group $\Br(K)$ is always torsion, so if we can prove $\oplus_{i=1}^n \Cl(R_i^{\text{h}})$ is as well, the result follows.  Let $Y \rightarrow X$ be a desingularization. We focus on the base change $\pi: Y^{\text{h}} \rightarrow \Spec(R^{\text{h}})$ to the henselian local ring at just one singular point $p$.  Let $E = \pi^{-1}(p)$ be the scheme-theoretic fiber.  We may choose $Y$ such that $E_{\text{red}}$ is a union of irreducible curves $F_j$ which are smooth and meet pairwise transversely, with no three containing a common point.

The following argument is due to Artin \cite{contract}. Suppose $G$ is the free abelian group of divisors supported on $E$, and consider the map $\alpha: \Pic(Y^{\text{h}}) \rightarrow \Hom(G,\Z)$ given by $L \mapsto (D \mapsto D \cdot L)$.  This restricts to a map $\alpha|_G: G \rightarrow \Hom(G,\Z)$ that is injective because the intersection matrix of the curves $F_j$ is negative definite.  Since $G$ and $\Hom(G,\Z)$ are free abelian groups of equal rank, $G \rightarrow \Hom(G,\Z)$ also has finite cokernel.  This allows us to find an effective Cartier divisor $Z = \sum_j r_j F_j$ with all $r_j > 0$ such that $\alpha(Z) = \alpha(-H)$ for an ample line bundle $H$ on $Y^{\text{h}}$ \cite[p. 491]{contract}.  If we restrict this Cartier divisor to the scheme associated to $Z$, the resulting line bundle $\mathcal{O}_Z(-Z)$ has positive degree on every irreducible component of $Z$, so it is ample.  We'll examine infinitesimal neighborhoods of the closed subscheme $Z$ in $Y^{\text{h}}$.

As before, for every $n \geq 1$, there is an exact sequence

$$\cdots \rightarrow H^1(Z,\mathcal{O}_Z(-nZ)) \rightarrow \Pic(Z_{n+1}) \rightarrow \Pic(Z_n) \rightarrow H^2(Z,\mathcal{O}_Z(-nZ)) \rightarrow \cdots.$$

Because $\dim(Z) = 1$, the last group is always zero.  Since $\mathcal{O}_Z(-Z)$ is ample, the first group is zero for $n \gg 0$ by Serre vanishing, which holds on any projective scheme \cite[Theorem II.5.2]{hartshorne}.  Therefore, the inverse limit $\varprojlim_n \Pic(Z_n)$ is constructed as a finite series of extensions of $\Pic(Z)$ by finite-dimensional $k$-vector spaces.  Applying Artin approximation, we have that $\Pic(Y^{\text{h}}) \rightarrow \varprojlim_n \Pic(E_n)$ is injective with dense image.  However, for large $n$, the scheme $E_n$ is nested between two infinitesimal neighborhoods of $Z$, where all restrictions of Picard groups are surjective (use a similar exact sequence to the above, e.g. \cite[09NY]{stacks}).  It follows that $\varprojlim_n \Pic(E_n) \cong \varprojlim_n \Pic(Z_n)$ and that both have the discrete topology, so $\Pic(Y^{\text{h}}) \cong \varprojlim_n \Pic(Z_n)$.  

Next, let $\bar{Z}$ be the disjoint union of the schemes $r_j F_j$, where $r_j F_j$ is the subscheme of $Y^{\text{h}}$ cut out by the ideal sheaf of $F_j$ to the power $r_j$.  Then $f: \bar{Z} \rightarrow Z$ is a finite map that is an isomorphism away from the finite set of intersection points and such that $\mathcal{O}_Z \subset f_* \mathcal{O}_{\bar{Z}}$.  It follows (see \cite[0C1M, 0C1N]{stacks}) that $\Pic(Z)$ is a finite sequence of extensions of $\Pic(\bar{Z})$ by quotients of $(k,+)$ or $(k,*)$.  Lastly, $\Pic(\bar{Z}) = \oplus_j \Pic(r_j F_j)$, where each summand is built from finite-dimensional $k$-vector spaces and $\Pic(F_j) \cong \Z \oplus \Pic^0(F_j)$.  Since the $\Pic^0(F_j)$ are groups of $k$-points of abelian varieties over $k$, they are all torsion.

\begin{commentA}
Over large fields of characteristic $p$, the key to non-torsion Brauer groups is some non-simple-connectedness of the exceptional locus, either via a cycle of curves (which gives a $k^*$ because you "link up" two points on the same connected component) or via $\Pic(C)$ for $C$ of positive genus.
\end{commentA}

Taken together, all of this implies that $G$ and $\Pic(Y^{\text{h}})$ have equal rank.  Since $G \rightarrow \Hom(G,\Z)$ is injective, the first map in the excision sequence of class groups $0 \rightarrow G \rightarrow \Pic(Y^{\text{h}}) \rightarrow \Cl(R^{\text{h}}) \rightarrow 0$ is injective also. Therefore, the quotient $\Cl(R^{\text{h}})$ is a torsion group, as desired.
\end{proof}

\end{document}